\documentclass[11pt]{article}
\usepackage{geometry}
\usepackage{color}
\usepackage{float}
\usepackage{textcomp}
\usepackage{amsthm}
\usepackage{amsmath}
\usepackage{amssymb}%
\usepackage{multirow}
\usepackage{changes}
\usepackage{xcolor}

\DeclareMathOperator{\grad}{grad}
\DeclareMathOperator{\hess}{hess}

\usepackage{mathtools}
\usepackage{booktabs}    
\usepackage{subfigure} 
\usepackage{caption} 
\usepackage{multirow}
\usepackage{multicol}    
\usepackage{array}       
\usepackage{graphicx}
\graphicspath{{figuras/}}
\usepackage{empheq,amsmath}
\usepackage[framemethod=tikz]{mdframed}

\usepackage
[ 
lined, 
boxruled, 
commentsnumbered
]
{algorithm2e}
\DecMargin{0.1cm} 

\newtheorem{theorem}{Theorem}[section]
\newtheorem{lemma}[theorem]{Lemma}

\newcommand{\bi}{\begin{itemize}}
\newcommand{\ei}{\end{itemize}}
\newcommand{\ba}{\begin{array}}
\newcommand{\ea}{\end{array}}

\begin{document}

\title{
A Riemannian AdaGrad-Norm Method}


\author{Glaydston C. Bento and Geovani N. Grapiglia and Mauricio S. Louzeiro and Daoping Zhang}

\author{  Glaydston de C. Bento  \thanks{Universidade Federal de Goi\'as, IME, Avenida Esperan\c{c}a, s/n, Campus Samambaia, CEP 74690-900, Goi\^ania, GO, Brazil. e-mail: {\tt glaydston@ufg.br}.}
\and
Geovani N. Grapiglia \thanks{Université Catholique de Louvain, ICTEAM/INMA, Avenue Georges Lema\^{\i}tre, 4-6/L4.05.01, B-1348, Louvain-la-Neuve, Belgium. e-mail: {\tt geovani.grapiglia@uclouvain.be}.}
 \and
Mauricio S. Louzeiro\thanks{Universidade Federal de Goi\'as, IME, Avenida Esperan\c{c}a, s/n, Campus Samambaia, CEP 74690-900, Goi\^ania, GO, Brazil. e-mail: {\tt mauriciolouzeiro@ufg.br}. }
\and
Daoping Zhang\thanks{Nankai University, School of Mathematical Sciences and LPMC, Tianjin 300071, China. e-mail: \tt{daopingzhang@nankai.edu.cn}.}
}

\maketitle

\begin{abstract}
We propose a manifold AdaGrad-Norm method (\textsc{MAdaGrad}), which extends the norm version of AdaGrad (AdaGrad-Norm) to Riemannian optimization.
 In contrast to line-search schemes, which may require several exponential map computations per iteration, \textsc{MAdaGrad} requires only one. Assuming the objective function $f$ has Lipschitz continuous Riemannian gradient, we show that the method requires at most $\mathcal{O}(\varepsilon^{-2})$ iterations to compute a point $x$ such that $\|\operatorname{grad} f(x)\|\leq \varepsilon$. Under the additional assumptions that $f$ is geodesically convex and the manifold has sectional curvature bounded from below, we show that the method takes at most $\mathcal{O}(\varepsilon^{-1})$ to find $x$ such that $f(x)-f_{low}\leq\epsilon$, where $f_{low}$ is the optimal value. Moreover, if $f$ satisfies the Polyak--\L{}ojasiewicz condition globally on the manifold, we establish a complexity bound of $\mathcal{O}(\log(\varepsilon^{-1}))$, provided that the norm of the initial Riemannian gradient is sufficiently large. For the manifold of symmetric positive definite matrices, we construct a family of nonconvex functions satisfying the PL condition. Numerical experiments illustrate the remarkable performance of \textsc{MAdaGrad} in comparison with Riemannian Steepest Descent equipped with Armijo line-search.

\end{abstract}

\noindent{\bf Key words:} Riemannian Optimization · Gradient Method · Adaptive Methods
· Worst-Case · Complexity Bounds\\
\noindent{\bf AMS subject classification:} 65K05  · 68Q25  ·90C30  · 49M37

\maketitle

\section{Introduction}

\subsection{Motivation and Contributions}

In this work we consider the minimization of a differentiable function $f:M\to\mathbb{R}$, where $M$ is a Riemannian manifold \cite{AbsilMahonySepulchre2008,Hu,boumal2023introduction}. Problems of this type appear in many important applications such as Low-Rank Matrix Completion \cite{Vandereycken}, Dictionary Learning \cite{Cherian} and Independent Component Analysis \cite{Selvan}. Many optimization algorithms originally developed for the Euclidean setting (\(M = \mathbb{R}^n\)) have been extended to Riemannian optimization. Notable examples include variants of the gradient method \cite{Smith1994,cruz1998geodesic}, Newton and conjugate gradient methods \cite{Smith1994}, quasi-Newton methods \cite{Huang}, trust-region methods \cite{Absil}, and cubic regularization of Newton’s method \cite{Agarwal}, among others. Special attention has been devoted to adaptive methods, which automatically select suitable stepsizes, trust-region radii, or regularization parameters without requiring prior knowledge of the problem-specific constants.
 
A classical example of an adaptive scheme is the gradient method with Armijo line search, which defines the iterates as  
\[
x_{k+1} = \exp_{x_{k}}\left(-\alpha_k \omega^{\ell_k} \operatorname{grad} f(x_{k})\right),
\]  
where \(\ell_k\) is the smallest nonnegative integer \(\ell\) such that  
\begin{equation}
f\left(\exp_{x_{k}}\left(-\alpha_k \omega^{\ell} \operatorname{grad} f(x_{k})\right)\right) \leq f(x_{k}) - \rho \alpha_k \omega^{\ell} \|\operatorname{grad} f(x_{k})\|^2,
\label{eq:armijo}
\end{equation}  
with \(\rho, \omega\in (0,1)\) and \(\alpha_0 > 0\) being user-defined parameters. In practice, the values \(\ell = 0, 1, 2, \dots\) are tested sequentially until inequality~\eqref{eq:armijo} is satisfied. This backtracking procedure may require multiple evaluations of the exponential map \(\exp_{x_{k}}(\cdot)\), which can make the method computationally expensive. To mitigate this issue, the RWNGrad method was recently proposed in \cite{grapiglia2023adaptive} as a Riemannian counterpart of the WNGrad method, originally developed for Euclidean optimization in \cite{ward2020adagrad}. Specifically, RWNGrad sets 
\begin{equation}
\left\{\begin{array}{rcl} x_{k+1}&=&\text{exp}_{x_{k}}\left(-\dfrac{1}{\beta_{k}}\text{grad}f(x_{k})\right),\quad \text{$\beta_{0}>0$}\\
\beta_{k+1}&=&\beta_{k}+\dfrac{\|\text{grad}f(x_{k})\|^{2}}{\beta_{k}},
\end{array}
\right.
\label{eq:rwngrad}
\end{equation}
thus requiring a single evaluation of the exponential map at each iteration. It was proved in \cite{grapiglia2023adaptive} that RWNGrad needs no more than $\mathcal{O}\left(\epsilon^{-2}\right)$ iterations to find $x_{k}$ such that $\|\text{grad}f(x_{k})\|\leq\epsilon$. More importantly, numerical experiments showed that RWNGrad is significantly faster than the Gradient Method with Armijo line search on problems over the manifold $M=\mathbb{P}^{n}_{++}$ of $(n\times n)$ symmetric and positive definite (SPD) matrices.

Motivated by the encouraging numerical performance of RWNGrad, in this paper we investigate a related yet distinct adaptive strategy, aiming to obtain a Riemannian algorithm with improved numerical performance and stronger theoretical guarantees compared to RWNGrad. Specifically, we propose MAdaGrad (Manifold AdaGrad-Norm), a Riemannian extension of the AdaGrad-Norm method \cite{ward2020adagrad}. In contrast to RWNGrad~\eqref{eq:rwngrad}, MAdaGrad sets
\begin{equation}
\left\{\begin{array}{rcl} 
\beta_{k+1}&=&\beta_{k}+\|\text{grad}f(x_{k})\|^{2},\quad \text{$\beta_{0}=0$,}\\
x_{k+1}&=&\text{exp}_{x_{k}}\left(-\dfrac{\eta}{\sqrt{\beta_{k+1}}}\text{grad}f(x_{k})\right),
\end{array}
\right.
\label{eq:MAdaGrad}
\end{equation}
with $\eta>0$ being a user-defined parameter. Regarding the theoretical guarantees of MAdaGrad~\eqref{eq:MAdaGrad}, we establish iteration-complexity bounds under various assumptions. When the objective function \( f(\cdot) \) is nonconvex, the method achieves a complexity of \( \mathcal{O}(\epsilon^{-2}) \). This rate improves to \( \mathcal{O}(\epsilon^{-1}) \) when \( f(\cdot) \) is convex and the manifold \( M \) has sectional curvature bounded below by a negative constant, or when \( f(\cdot) \) is possibly nonconvex but satisfies the Polyak--Łojasiewicz (PL) condition globally. In addition, we provide a family of nonconvex functions over the manifold of symmetric positive definite matrices \( \mathbb{P}^{n}_{++} \) that satisfy the PL condition. Finally, our numerical experiments demonstrate that MAdaGrad can significantly outperform both RWNGrad and the gradient method with Armijo line search.

\subsection{Related Literature}

In recent years, the challenge of training machine learning models has motivated the development and analysis of numerous adaptive variants of the Stochastic Gradient Descent (SGD) method, including AdaGrad \cite{duchi2011adaptive}, RMSProp \cite{tieleman2012rmsprop}, Adam \cite{kingma2015adam}, and AMSGrad \cite{reddi2018convergence}. A key feature of these methods is the use of distinct stepsizes for updating each component of the iterate. For example, the batch version of AdaGrad applied to the minimization of $f:\mathbb{R}^{n}\to\mathbb{R}$ defines the iterates by
\begin{equation*}
\left\{\begin{array}{rcl} \beta_{k+1}&=&\beta_{k}+\nabla f(x_{k})\odot\nabla f(x_{k}),\quad \beta_{k}=0\in\mathbb{R}^{n},\\
& &\\
\left[x_{k+1}\right]_{i}&=&\left[x_{k}\right]_{i}+\frac{\eta}{\sqrt{\left[\beta_{k+1}\right]_{i}}}[\nabla f(x_{k})]_{i},\quad i=1,\ldots,n,
\end{array}
\right.
\end{equation*}
where $\left[\nabla f(x_{k})\odot\nabla f(x_{k})\right]_{i}=\left[\nabla f(x_{k})\right]_{i}^{2}$ for $i=1,\ldots,n$. The component-wise nature of adaptive methods such as AdaGrad complicates their extension to the Riemannian setting, due to the lack of intrinsic coordinate systems on general manifolds. In \cite{Becigneul}, this issue was addressed by considering the special case where $M$ is a Cartesian product of Riemannian manifolds. By exploiting this additional structure, the authors proposed Riemannian extensions of AdaGrad, Adam, and AMSGrad. A different approach was considered in \cite{Sakai}, where the authors presented a Riemannian adaptive method that encompasses extensions of RMSProp, Adam, and AMSGrad for the case where $M$ is an embedded submanifold of Euclidean space. Therefore, given the component-wise nature of this class of adaptive methods, their generalization to the Riemannian setting typically requires additional assumptions on the manifold $M$, which restricts their applicability. For this reason, in the present work we focus on generalizing AdaGrad-Norm \cite{ward2020adagrad}, which, similar to WNGrad \cite{wu2018wngrad}, does not rely on a coordinate system and can thus be extended to general Riemannian optimization problems.

\subsection{Contents}

The remainder of the paper is organized as follows. In Section~\ref{sec:preliminaries}, we introduce the necessary concepts and notations from Riemannian geometry. Section~\ref{sec:algorithm} presents the MAdaGrad algorithm and establishes key auxiliary results concerning its iterates and step sizes. In Section~\ref{sec:complexity}, we derive iteration-complexity bounds for the nonconvex, convex, and PL cases. We also provide a family of nonconvex functions that satisfy the PL property globally. Finally, in Section~\ref{sec:numerics} we report numerical results.

\section{Preliminary}
\label{sec:preliminaries}
In this section, we recall  some  concepts, notations, and basic results  about Riemannian manifolds.   For more details   see, for example, \cite{doCarmo1992,Sakai1996,UdristeLivro1994,Rapcsak1997}.

 We denote by $T_pM$ the {\it tangent space} of a Riemannian manifold $M$ at $p$.  The corresponding norm associated to the Riemannian metric $\langle \cdot ~, ~ \cdot \rangle$ is denoted by $\|  ~\cdot~ \|$. We use $\ell(\gamma)$ to denote the length of a piecewise smooth curve $\gamma:[a,b]\rightarrow M$. The Riemannian distance between $p$ and $q$ in a finite-dimensional Riemannian manifold $M$ is denoted by $d(p,q)$. This distance induces the original topology on $M$, so that $(M,d)$ becomes a complete metric space.
 Let $( N,  \langle \cdot ~, ~ \cdot \rangle)$ and $(M,   \langle\!\!\langle \cdot ~, ~ \cdot \rangle\!\!\rangle)$ be Riemannian manifolds and $\Phi : N \to M$ be an isometry, that is, $\Phi$  is $C^{\infty}$, and for all $q \in N$ and $u, v \in T_qN$, we have $\langle u , v \rangle=\langle\!\!\langle d\Phi_q u , d\Phi_q v \rangle\!\!\rangle $, where $d\Phi_q : T_{q}N \to T_{\Phi(q)}M$ is the differential of $\Phi$ at $q\in N$.  One can verify that $\Phi$  preserves geodesics, that is, $\beta$ is a geodesic in $N$ if and only if $\Phi\circ \beta$ is a geodesic in $M$.   Denote by ${\cal X}(M)$, the space of smooth vector fields on $M$. Let $\nabla$ be the Levi-Civita connection associated to $(M, \langle \cdot ~, ~ \cdot \rangle)$. The Riemannian metric induces a mapping  $f\mapsto\grad f $ that associates  each   differentiable function to its {\it gradient} via the rule $\langle\grad f,X\rangle=d f(X)$, for all $ X\in{\cal X}(M)$.      A vector field $V$ along $\gamma$ is said to be {\it parallel} iff $\nabla_{\gamma^{\prime}} V=0$. If $\gamma^{\prime}$ itself is parallel, we say that $\gamma$ is a {\it geodesic}. Given that the geodesic equation $\nabla_{\ \gamma^{\prime}} \gamma^{\prime}=0$ is a second order nonlinear ordinary differential equation, then the geodesic $\gamma=\gamma _{v}( \cdot ,p)$ is determined by its position $p$ and velocity $v$ at $p$. It is easy to check that $\|\gamma ^{\prime}\|$ is constant. The restriction of a geodesic to a  closed bounded interval is called a {\it geodesic segment}. A geodesic segment joining $p$ to $q$ in $ M$ is said to be {\it minimal} if its length is equal to $d(p,q)$. For each $t \in [a,b]$, $\nabla$ induces an isometry, relative to $ \langle \cdot , \cdot \rangle  $, $P_{\gamma,a,t} \colon T _{\gamma(a)} {M} \to T _
{\gamma(t)} {M}$ defined by $ P_{\gamma,a,t}\, v = V(t)$, where $V$ is the unique vector field on $\gamma$ such that
$ \nabla_{\gamma'(t)}V(t) = 0$ and $V(a)=v$, the so-called {\it parallel transport} along  the geodesic segment   $\gamma$ joining  $\gamma(a)$ to $\gamma(t)$.  When there is no confusion, we consider the notation $P_{\gamma,p,q}$  for the parallel transport along  the geodesic segment  $\gamma$ joining  $p$ to $q$.  A Riemannian manifold is {\it complete} if the geodesics are defined for any values of $t\in \mathbb{R}$. Hopf-Rinow's theorem asserts that any pair of points in a  complete Riemannian  manifold $M$ can be joined by a (not necessarily unique) minimal geodesic segment. A set  $\Omega\subseteq M$ is said to be {\it convex}  iff any geodesic segment with end points in $\Omega$ is contained in $\Omega$.  A function $f:M\to\mathbb{R}$  is said to be  {\it convex} on a convex set $\Omega $ iff for any geodesic segment $\gamma:[a, b]\to\Omega$, the composition $f\circ\gamma:[a, b]\to\mathbb{R}$ is convex.
Owing to  the completeness of the Riemannian manifold $M$, the {\it exponential map} $\mbox{exp}_{p}:T_{p}  M \to M $ can be  given by $\exp_{p}v\,=\, \gamma _{v}(1,p)$, for each $p\in M$. A complete, simply connected Riemannian manifold of non-positive sectional curvature is called a {\it Hadamard manifold}.  For all $p\in M$, the  exponential map $\mbox{exp}_{p}:T_{p}  M \to M $  is  a diffeomorphism and $\mbox{exp}^{-1}_{p}:M\to T_{p}M$ denotes its inverse. In this case, $d({q}\, , \, p)\,=\,||\mbox{exp}^{-1}_{p}q||$ and   the function $d_{q}^2: M\to\mathbb{R}$ defined by $ d_{q}^2(p):=d^2(q,p)$ is  $C^{\infty}$ and $ \grad d_{q}^2(p):=-2\mbox{exp}^{-1}_{p}{q}.$

 {\it In this paper, all manifolds are assumed to be  connected,   finite dimensional, and complete.}
 
\section{A Riemannian AdaGrad-Norm Method}
\label{sec:algorithm}

Consider the problem
\begin{equation}
\min_{x\in M}\,f(x),
\label{eq:25}
\end{equation}
where $M$ is a Riemannian manifold and $f\colon M\to\mathbb{R}$ is a differentiable function. 
As introduced in \cite{cruz1998geodesic}, given $L\geq 0$, the gradient vector fields $\grad f$ is said to be $L$-Lipschitz continuous if,   for any  points $ p$ and $q\in M$ and $\gamma$,  a  geodesic  segment joining $p$ to $q$, one has
$\left\|P_{\gamma,p,q} \mbox{grad} f(p)- \mbox{grad} f(q)\right\|\leq L d(p, q).$ 
\\[0.2cm]
\noindent Let us assume that:
\vspace{0.2cm}
\begin{mdframed}
\begin{itemize}
\item[\textbf{A1.}] $\grad f$ is $L$-Lipschitz continuous;
\item[\textbf{A2.}] $f$ has a global minimizer, with optimal value denoted by $f^{*}$.
\end{itemize}
\end{mdframed}

\newpage

\noindent Below we propose a Riemannian generalization of the batch version of method AdaGrad-Norm \cite{ward2020adagrad,xie2020linear}.
\begin{mdframed}
\noindent\textbf{Algorithm 1.} Riemannian AdaGrad-Norm (MAdaGrad).
\\[0.2cm]
\noindent\textbf{Step 0.} Given $x_{0}\in M$ and $\eta>0$, set $\beta_{0}=0$ and $k:=0$. 
\\[0.2cm]
\noindent\textbf{Step 1.} If $\grad f(x_{k})=0$, then {\bf stop}; otherwise, compute
\begin{align}
\beta_{k+1}&= \beta_{k}+\|\text{grad}f(x_{k})\|^{2},\label{eq:r1} \\
\alpha_{k}&=\frac{\eta}{\sqrt{\beta_{k+1}}},\label{eq:r2}\\
x_{k+1}&=\text{exp}_{x_{k}}\left(-\alpha_{k}\text{grad} f(x_{k})\right).\label{eq:r3}
\end{align}
\noindent\textbf{Step 2.} Set $k:=k+1$ and go to Step 1.
\end{mdframed}
The next lemma is a consequence of \cite[Lemma 5.1]{cruz1998geodesic} and has appeared in \cite{bento2017iteration}.
\begin{lemma}
Suppose that A1 holds. Then, 
\begin{equation}\label{eq:1001}
f( \exp_{p}v) \leq f(p) + \langle {\emph{grad}} f(p), v\rangle+\frac{L}{2}\left\|v\right\|^{2}, \qquad  p\in M, \quad v\in T_pM .
\end{equation}
\end{lemma}

\begin{lemma}\label{le:rlc}
Suppose that A1 holds and let $\left\{x_{k}\right\}_{k\geq 0}$ be generated by Algorithm 1. Then, the following hold:
\begin{enumerate}
\item [i)] $f(x_{k+1})\leq f(x_{k})+\dfrac{L\alpha_{k}^{2}}{2}(\beta_{k+1}-\beta_{k})$ for all $ k\geq 0$;
\item [ii)] if $\alpha_{k}\leq 1/L$ for some $k\geq 0$, then $f(x_{k})-f(x_{k+1})\geq\dfrac{\alpha_{k}}{2}\|\emph{grad} f(x_{k})\|^{2}$;
\item [iii)] if $\beta_{k+1}<\eta^{2}L^{2}$ for all $k=0,\ldots,k_{0}-1$, where $k_{0}\geq 1$, then
\begin{equation}
\alpha_{k}>\dfrac{1}{L},\quad\,\,k=0,\ldots,k_{0}-1,
\label{eq:9.1}
\end{equation}
\begin{equation}
\sum_{k=0}^{k_{0}-1}\alpha_{k}\|\emph{grad} f(x_{k})\|^{2}\leq\dfrac{\eta^{3}L^{2}}{\|\emph{grad} f(x_{0})\|},
\label{eq:10.1}
\end{equation}
\begin{equation}
f(x_{k_{0}})\leq f(x_{0})+\dfrac{\eta^{4}L^{3}}{2\|\emph{grad} f(x_{0})\|^{2}}.
\label{eq:extra1.1}
\end{equation}
\end{enumerate}
\end{lemma}
\begin{proof}
By using  \eqref{eq:1001} with $p = x_{k}$ and $v = -\alpha_k \grad f(x_{k})$, and  \eqref{eq:r3}, we obtain
\begin{eqnarray}\label{eq:53}
f(x_{k+1})&\leq & f(x_{k})-\alpha_{k}\|\mbox{grad} f(x_{k})\|^{2}+\dfrac{L\alpha_{k}^{2}}{2}\|\mbox{grad} f(x_{k})\|^{2}\nonumber\\
&\leq & f(x_{k})+\dfrac{L\alpha_{k}^{2}}{2}\|\mbox{grad} f(x_{k})\|^{2}
\label{eq:53}
\end{eqnarray}
for all $k\geq 0$.  The proof of item $i)$  follows by combining the last inequality in \eqref{eq:53} with \eqref{eq:r1}, and the proof of item $ii)$ follows by using the first inequality in \eqref{eq:53} along with the fact that $-L\alpha_k\geq -1$ for some $k\geq 0$. Now, let us assume that $\beta_{k+1}<\eta^{2}L^{2}$ for $k=0,\ldots,k_{0}-1$, for some $k_{0}\geq 1$. Consequently, we have $\sqrt{\beta_{k+1}}<\eta L$ for $k=0,\ldots,k_{0}-1$, and the proof of  \eqref{eq:9.1} is an immediate consequence of  \eqref{eq:r2}. 
On the other hand, by using  \eqref{eq:r2}  again, we obtain
\begin{eqnarray}
\sum_{k=0}^{k_{0}-1}\alpha_{k}\|\mbox{grad} f(x_{k})\|^{2}&=&\sum_{k=0}^{k_{0}-1}\dfrac{\eta}{\sqrt{\beta_{k+1}}}\|\mbox{grad} f(x_{k})\|^{2}\nonumber\\
&\leq &\dfrac{\eta}{\|\mbox{grad} f(x_{0})\|}\sum_{k=0}^{k_{0}-1}\|\mbox{grad} f(x_{k})\|^{2},
\label{eq:53.1}
\end{eqnarray}
where the last inequality follows from the fact that 
\begin{equation*}
\beta_{k+1}\geq \beta_1=\beta_0+\|\mbox{grad} f(x_{0})\|^2=\|\mbox{grad} f(x_{0})\|^2,
\end{equation*}
which is an immediate consequence of \eqref{eq:r1} and $\beta_0=0$. Combining the inequality in \eqref{eq:53.1} with \eqref{eq:r1}, we get
\begin{equation}\label{eq:56}
\sum_{k=0}^{k_{0}-1}\alpha_{k}\|\mbox{grad} f(x_{k})\|^{2}\leq \dfrac{\eta}{\|\mbox{grad} f(x_{0})\|}\sum_{k=0}^{k_{0}-1}(\beta_{k+1}-\beta_{k})=\dfrac{\eta}{\|\mbox{grad} f(x_{0})\|}\beta_{k_{0}}.
\end{equation}
Hence, \eqref{eq:10.1} is obtained directly by combining \eqref{eq:56} with the inequality $\beta_{k_0} < \eta^2 L^2$, which follows by taking $k = k_{0} - 1$ in $\sqrt{\beta_{k+1}} < \eta L$. To conclude the proof of item $iii)$, note that
\begin{equation*}
f(x_{k_{0}}) - f(x_{0}) = \sum_{k=0}^{k_{0}-1} \left( f(x_{k+1}) - f(x_{k}) \right) \leq \sum_{k=0}^{k_{0}-1} \dfrac{L \alpha_{k}^{2}}{2} (\beta_{k+1} - \beta_{k}),
\end{equation*}
where the last inequality is obtained from item $i)$. On the other hand, from \eqref{eq:r2}, we have $\alpha_{k}^{2} = \eta^2 / \beta_{k+1}$, which, combined with the last inequality and using again that $\beta_{k+1} \geq \|\mbox{grad} f(x_{0})\|^2$ for all $k \in \mathbb{N}$, yields
\begin{eqnarray*}
f(x_{k_{0}}) - f(x_{0}) &\leq & \dfrac{L \eta^{2}}{2} \sum_{k=0}^{k_{0}-1} \dfrac{1}{\beta_{k+1}} (\beta_{k+1} - \beta_{k})\\ 
&\leq & \dfrac{L \eta^{2}}{2 \|\mbox{grad} f(x_{0})\|^{2}} \sum_{k=0}^{k_{0}-1} (\beta_{k+1} - \beta_{k})\\
&=& \dfrac{L \eta^{2}}{2 \|\mbox{grad} f(x_{0})\|^{2}} \beta_{k_{0}}.
\end{eqnarray*}
Therefore, \eqref{eq:extra1.1} follows by using the fact that $\beta_{k_0} < \eta^2 L^2$, concluding the proof. 
\end{proof}

\begin{lemma}\label{lem:4.1}
Suppose that A1 and A2 hold and let $\left\{x_{k}\right\}_{k\geq 0}$ be generated by Algorithm 1. If 
\begin{equation}\label{eq:r12}
k_{0}=\inf \left\{k\in\mathbb{N}\colon\beta_{k+1}\geq\eta^{2}L^{2}\right\} <+\infty,
\end{equation}
then
\begin{equation}
\sum_{k=k_{0}}^{T-1}\alpha_{k}\|\grad f(x_{k})\|^{2}\leq 2\left(f(x_{0})-f^{*}+\dfrac{\eta^{4}L^{3}}{2\|\grad f(x_{0})\|^{2}}\right), \qquad \forall \, T>k_{0},
\label{eq:r14}
\end{equation}
and
\begin{equation}
\alpha_{k}\geq \left(L+\frac{2(f(x_{0})-f^{*})}{\eta^{2}}+\frac{\eta^{2}L^{3}}{\|\grad f(x_{0})\|^{2}}\right)^{-1},\qquad \forall \, k\geq k_{0}.
\label{eq:r13}
\end{equation}
\end{lemma}

\begin{proof}
In view of \eqref{eq:r2}, \eqref{eq:r1}, and \eqref{eq:r12}, we have
\begin{equation*}
\alpha_k =\dfrac{\eta}{\sqrt{\beta_{k+1}}} \leq \dfrac{\eta}{\sqrt{\beta_{k_0+1}}}\leq\dfrac{1}{L},\qquad k\geq k_{0}.
\end{equation*}
Thus, by Lemma \ref{le:rlc} (ii),  it follows that
\begin{equation*}
f(x_{k})-f(x_{k+1})\geq\dfrac{\alpha_{k}}{2}\|\grad f(x_{k})\|^{2},\qquad k\geq k_{0}.
\end{equation*}
Summing up these inequalities for $k=k_{0},\ldots,T-1$, and using A2, we get
\begin{equation}
\sum_{k=k_{0}}^{T-1}\dfrac{\alpha_{k}}{2}\|\grad f(x_{k})\|^{2}\leq f(x_{k_{0}})-f^{*}.
\label{eq:15}
\end{equation}
If $k_{0}=0$, then it follows from (\ref{eq:15}) that \eqref{eq:r14} is true. If $k_{0}\geq 1$, then 
$
\beta_{k+1}<\eta^{2}L^{2}
$
for $k=0,\ldots,k_{0}-1$.
Thus, by inequality \eqref{eq:extra1.1} in Lemma  \ref{le:rlc} we have
\begin{equation}
f(x_{k_{0}})-f^{*}\leq f(x_{0})-f^{*}+\dfrac{\eta^{4}L^{3}}{2\|\grad f(x_{0})\|^{2}}.
\label{eq:16}
\end{equation}
Therefore, combining (\ref{eq:15}) and (\ref{eq:16}) we see that (\ref{eq:r14}) is also true when $k_{0}\geq 1$. On the other hand, notice that
\begin{align}
\dfrac{\alpha_{k}}{2}\|\grad f(x_{k})\|^{2} & =\dfrac{\eta}{2}\dfrac{\beta_{k+1}-\beta_{k}}{\sqrt{\beta_{k+1}}} \nonumber\\
& =\dfrac{\eta}{2}\dfrac{(\sqrt{\beta_{k+1}}-\sqrt{\beta_{k}})(\sqrt{\beta_{k+1}}+\sqrt{\beta_{k}})}{\sqrt{\beta_{k+1}}} \geq \dfrac{\eta}{2}(\sqrt{\beta_{k+1}}-\sqrt{\beta_{k}}).
\label{eq:17}
\end{align}
Now, combining \eqref{eq:r14} and (\ref{eq:17}), it follows that
\begin{eqnarray*}
\dfrac{\eta}{2}(\sqrt{\beta_{T}}-\sqrt{\beta_{k_{0}}})&=&\sum_{k=k_{0}}^{T-1} \dfrac{\eta}{2} ( \sqrt{\beta_{k+1}}-\sqrt{\beta_{k} }) \leq \sum_{k=k_{0}}^{T-1}\dfrac{\alpha_{k}}{2}\|\grad f(x_{k})\|^{2}\\
&\leq & f(x_{0})-f^{*}+\dfrac{\eta^{4}L^{3}}{2\|\grad f(x_{0})\|^{2}},
\end{eqnarray*}
which implies
\begin{align*}
\sqrt{\beta_{T}}&\leq  \sqrt{\beta_{k_{0}}}+\dfrac{2(f(x_{0})-f^{*})}{\eta}+\dfrac{\eta^{3}L^{3}}{\|\grad f(x_{0})\|^{2}}\\
                  &\leq \eta L+\dfrac{2(f(x_{0})-f^{*})}{\eta}+\dfrac{\eta^{3}L^{3}}{\|\grad f(x_{0})\|^{2}}.
\end{align*}
Since $T$ is an arbitrary integer bigger than $k_{0}$, we have
\begin{equation*}
\sqrt{\beta_{k+1}}\leq\eta L+\dfrac{2(f(x_{0})-f^{*})}{\eta}+\dfrac{\eta^{3}L^{3}}{\|\grad f(x_{0})\|^{2}},\qquad k\geq k_{0}.
\end{equation*}
Consequently,
\begin{equation*}
\alpha_{k}=\dfrac{\eta}{\sqrt{\beta_{k+1}}}\geq \left( L+\dfrac{2(f(x_{0})-f^{*})}{\eta^2}+\dfrac{\eta^{2}L^{3}}{\|\grad f(x_{0})\|^{2}}\right)^{-1},\qquad k\geq k_{0},
\end{equation*}
which implies that \eqref{eq:r13} is true.
\end{proof}

\begin{lemma}
\label{lem:5}
Suppose that A1 and A2 hold and let $\left\{x_{k}\right\}_{k\geq 0}$ be generated by Algorithm 1. Then,
\begin{equation}
\alpha_{k}\geq \left(L+\frac{2(f(x_{0})-f^{*})}{\eta^{2}}+\frac{\eta^{2}L^{3}}{\|\grad f(x_{0})\|^{2}}\right)^{-1} \equiv \alpha_{\min}, \qquad k\geq 0,
\label{eq:18}
\end{equation}
and 
\begin{equation}
\sum_{k=0}^{T-1}\alpha_{k}\|\grad f(x_{k})\|^{2}\leq  \dfrac{ \eta^{3}L^{2}  }{\|\grad f(x_{0})\|} + 2(f(x_{0})-f^{*})+\dfrac{\eta^{4}L^{3}   }{\|\grad f(x_{0})\|^{2}}, \qquad T\geq 1.
\label{eq:19}
\end{equation}
\end{lemma}

\begin{proof}
Let us divide the proof in two cases.
\\[0.2cm]
\noindent\textbf{Case 1:} $ k_0=\inf \left\{k\in\mathbb{N}\colon\beta_{k+1}\geq\eta^{2}L^{2}\right\}=+\infty$.
\\[0.2cm]
In this case, we have
$\beta_{k+1}<\eta^{2}L^{2}$ for all $k\geq 0$. Thus, it follows from Lemma \ref{le:rlc} (iii) that 
\begin{equation*}
\alpha_{k}>\dfrac{1}{L},\quad\forall k\geq 0\quad\text{and}\quad \sum_{k=0}^{T-1}\alpha_{k}\|\grad f(x_{k})\|^{2}\leq\dfrac{\eta^{3}L^{2}}{\|\grad f(x_{0})\|},\qquad T\geq 1.
\end{equation*}
Therefore, (\ref{eq:18}) and (\ref{eq:19}) hold.
\\[0.2cm]
\noindent\textbf{Case 2:} $ k_0=\inf \left\{k\in\mathbb{N}\colon\beta_{k+1}\geq\eta^{2}L^{2}\right\}<+\infty$.
\\[0.2cm]
In this case, if $k_{0}=0$, then (\ref{eq:18}) and (\ref{eq:19}) follow directly from Lemma \ref{lem:4.1}. If $k_{0}\geq 1$, it follows from Lemmas \ref{le:rlc} (iii) and \ref{lem:4.1} that 
$\alpha_{k}>1/L $ for $k=0,\ldots,k_{0}-1$,
and
$\alpha_{k}\geq \alpha_{\min}$ for  $k\geq k_{0}$.
Therefore, (\ref{eq:18}) is true for all $k\geq 0$. Moreover, given $T\geq 1$, we have two possibilities. 
\\[0.2cm]
\noindent\textbf{Subcase 2.1:} $T\leq k_{0}$.
\\[0.2cm]
In this subcase, it follows from Lemma \ref{le:rlc} that
\begin{equation*}
\sum_{k=0}^{T-1}\alpha_{k}\|\grad f(x_{k})\|^{2}\leq\sum_{k=0}^{k_{0}-1}\alpha_{k}\|\grad f(x_{k})\|^{2}\leq\dfrac{\eta^{3}L^{2}}{\|\grad f(x_{0})\|},
\end{equation*}
and so (\ref{eq:19}) is true.
\\[0.2cm]
\noindent\textbf{Subcase 2.2:} $T>k_{0}$.
\\[0.2cm]
In this subcase, by Lemma \ref{le:rlc} (iii) and Lemma 3 \ref{lem:4.1} we have
\begin{align*}
\sum_{k=0}^{T-1}\alpha_{k}\|\grad f(x_{k})\|^{2}&=\sum_{k=0}^{k_{0}-1}\alpha_{k}\|\grad f(x_{k})\|^{2}+\sum_{k=k_{0}}^{T-1}\alpha_{k}\|\grad f(x_{k})\|^{2}\\
 &\leq  \dfrac{\eta^{3}L^{2}}{\|\grad f(x_{0})\|}+2\left(f(x_{0})-f^{*}\right)+\dfrac{\eta^{4}L^{3}}{\|\grad f(x_{0})\|^2},
\end{align*}
that is, (\ref{eq:19}) is true.
\end{proof}

\section{Worst-Case Complexity Bounds}
\label{sec:complexity}

In this section, we establish iteration-complexity bounds for MAdaGrad~\eqref{eq:MAdaGrad}. We show that the method achieves a complexity of $\mathcal{O}(\epsilon^{-2})$, which improves to $\mathcal{O}(\epsilon^{-1})$ both when the objective function is convex and when it globally satisfies the Polyak--Łojasiewicz (PL) condition. It is worth mentioning that, in the convex case, we assume that the manifold $M$ has sectional curvature bounded below by a negative constant.

\subsection{General Case}

\begin{theorem}
\label{thm:1}
Suppose that A1-A3 hold and let $\left\{x_{k}\right\}_{k\geq 0}$ be generated by Algorithm 1. Given $\epsilon>0$, let 
\begin{equation*}
    T_{g}(\epsilon)=\inf\left\{k\in\mathbb{N}\,:\,\|\grad f(x_{k})\|\leq\epsilon\right\}.
\end{equation*}
Then
\begin{equation}
T_{g}(\epsilon)\leq\left[ \dfrac{\eta^{3}L^{2}}{ \alpha_{\min} \|\grad f(x_{0})\|} + \frac{2(f(x_{0}) - f^{*})}{ \alpha_{\min} } + \frac{\eta^{4}L^{3}}{ \alpha_{\min} \|\grad f(x_{0})\|^{2}} \right] \epsilon^{-2},
\label{eq:complexity1}
\end{equation}
where $\alpha_{\min}$ is defined in (\ref{eq:18}).
\end{theorem}
\begin{proof}
If $T_{g}(\epsilon)=0$, then (\ref{eq:complexity1}) is true. Thus, let us assume that $T_{g}(\epsilon)\geq 1$. By Lemma \ref{lem:5}, we have
\begin{align*}
&\dfrac{\eta^{3}L^{2}}{\|\grad f(x_{0})\|} + 2(f(x_{0}) - f^{*}) + \dfrac{\eta^{4}L^{3}}{\|\grad f(x_{0})\|^{2}}\\
&\geq \sum_{k=0}^{T_{g}(\epsilon)-1} \alpha_{k} \|\grad f(x_{k})\|^{2} \\
&\geq \alpha_{\min} \sum_{k=0}^{T_{g}(\epsilon)-1} \|\grad f(x_{k})\|^{2} \\
&\geq \alpha_{\min} T_{g}(\epsilon)\epsilon^2.
\end{align*}
Then, isolating $T_{g}(\epsilon)$, we conclude that (\ref{eq:complexity1}) also holds in this case.
\end{proof}


\subsection{Convex Case}

\begin{lemma}\label{lem:boud}
Suppose that A1 and A2 hold, and  let $\{x_{k}\}_{k\geq 0}$ be  generated by  Algorithm 1. Then, 
\small
	\begin{equation}\label{desi.des}
	\sum_{k=0}^{\infty}\alpha^2_k\left\|\grad f(x_{k})\right\|^2\leq \rho \equiv  \frac{\eta}{\|\grad f(x_{0}) \|} \left[ \dfrac{ \eta^{3}L^{2}  }{\|\grad f(x_{0})\|} + 2(f(x_{0})-f^{*})+\dfrac{\eta^{4}L^{3}   }{\|\grad f(x_{0})\|^{2}} \right].
	\end{equation}
\normalsize
\end{lemma}
\begin{proof}
From \eqref{eq:r1} and \eqref{eq:r2}, we have  
$\alpha_{k} \leq \eta/\|\grad f(x_{0}) \|$ for all $k\geq 0$, which implies  
$$
\sum_{k=0}^{\infty}\alpha^2_k\left\|\grad f(x_{k})\right\|^2 \leq  \frac{\eta}{\|\grad f(x_{0}) \|} \sum_{k=0}^{\infty}  \alpha_k\left\|\grad f(x_{k})\right\|^2.
$$  
Thus, the proof of \eqref{desi.des}  follows from Lemma \ref{lem:5}.
\end{proof}
Now, let us consider the following assumptions:  
\begin{mdframed}
\begin{itemize}  
    \item[\textbf{A3.}]  \( M \) has sectional curvature bounded below by a negative constant, i.e., \( K \geq \kappa \) with \( \kappa < 0 \);  
    \item[\textbf{A4.}]  \( f: M \to \mathbb{R}\) is convex on $M$ and admits a minimizer $q$ with $f^{*}=f(q)$.  
\end{itemize}  
\end{mdframed}

Taking into account Lemma \ref{lem:boud}, the next lemma follows from \cite[Lemma 3.6]{ferreira2020iteration}.
\begin{lemma}\label{pr:ltd}
Suppose that A3 and A4 hold, and let $\{x_{k}\}_{k\geq 0}$ be the sequence generated by Algorithm 1. Then, for each $k\geq0$,  the following inequality holds:
\begin{equation*}\label{eq;desgen}
d^2(x_{k+1},q)\leq d^2(x_{k},q) +  {\cal K}_{\rho,\kappa}^q \alpha_k^2\left\| \grad f(x_{k}) \right\|^2 + 2\alpha_k[f^{*}-f(x_{k})], 
\end{equation*}
where 
\begin{equation}  \label{eq:kkappa}
{\cal K}_{\rho,\kappa}^q  := \frac{\sinh\left(\hat{\kappa}\sqrt{\rho}\right)}{\hat{\kappa}\sqrt{\rho}}\frac{{\cal C}_{\rho,\kappa}^q }{\tanh {\cal C}_{\rho,\kappa}^q } 
\qquad {\cal C}_{\rho,\kappa}^q  :=\cosh^{-1}\left(\cosh(\hat{\kappa}d(x_{0},q))e^{\frac{1}{2}\left(\hat{\kappa}\sqrt{\rho}\right)\sinh\left(\hat{\kappa}\sqrt{\rho}\right)}\right),
\end{equation}
with $\rho$ is defined in \eqref{desi.des} and  $\hat{\kappa}\equiv \sqrt{|\kappa|}$.
\end{lemma}

\begin{theorem}
\label{thm:2}
Suppose that A1--A4 hold, and let $\left\{x_{k}\right\}_{k\geq 0}$ be the sequence generated by Algorithm 1. Given \( \epsilon > 0 \), let
\begin{equation}
    T_{f}(\epsilon)=\inf\left\{k\in\mathbb{N}\,:\,f(x_{k})-f^{*}\leq\epsilon\right\}.
    \label{eq:hitting}
\end{equation}
Then
\begin{equation}
T_{f}(\epsilon) \leq \left( \frac{d^2(x_{0},q) + \rho {\cal K}_{\rho,\kappa}^q }{2 \alpha_{\min} } \right) \epsilon^{-1},
\label{eq:21}
\end{equation}
where $\alpha_{\min}$, \( \rho \) and \( {\cal K}_{\rho, \kappa}^q \) are defined in (\ref{eq:18}), \eqref{desi.des} and \eqref{eq:kkappa}, respectively.
\end{theorem}

\begin{proof}
If $T_{f}(\epsilon)=0$, then (\ref{eq:21}) is true. Thus, let us assume that $T_{f}(\epsilon)\geq 1$. By combining Lemma \ref{pr:ltd} with \eqref{eq:18}, we obtain  
\begin{equation*}  
f(x_{k}) - f^{*}  
\leq \frac{ d^2(x_{k},q) - d^2(x_{k+1},q) +  
{\cal K}_{\rho,\kappa}^q  
\alpha^{2}_{k} \|\grad f(x_{k})\|^{2} }{2 \alpha_{\min}},                
\end{equation*}  
for all \( k \geq 0 \). Summing this inequality over \( k = 0, \dots, T_{f}(\epsilon)-1 \) and applying Lemma \ref{lem:boud}, we obtain  
\begin{align*}  
\epsilon&<\min\left\{  f(x_{k}) - f^{*} \colon k=0,\ldots,T_{f}(\epsilon)-1  \right\}\leq \frac{1}{T_{f}(\epsilon)} \sum_{k=0}^{T_{f}(\epsilon)-1} (f(x_{k}) - f^{*}) \\  
& \leq \frac{1}{T_{f}(\epsilon)} \sum_{k=0}^{T_{f}(\epsilon)-1} \frac{ d^2(x_{k},q) - d^2(x_{k+1},q) +  
{\cal K}_{\rho,\kappa}^q  
\alpha^{2}_{k} \|\grad f(x_{k})\|^{2} }{2 \alpha_{\min}} \\  
& \leq \frac{1}{T_{f}(\epsilon)} \left( \frac{d^2(x_{0},q) + \rho {\cal K}_{\rho,\kappa}^q }{2\alpha_{\min}} \right).
\end{align*}  
Therefore, isolating $T_{f}(\epsilon)$ we conclude that (\ref{eq:21}) is true.  
\end{proof}

\subsection{$\mu$-Polyak-Lojasiewicz Case}

Throughout this section, the results are established taking into account the following assumption:
\begin{mdframed}
\begin{itemize}
\item[\textbf{A5.}] \( f: M \to \mathbb{R}\) has a minimizer $q\in M$, with $f^{*}=f(q)$, and there exists $\mu>0$ such that
\begin{equation*}
f(x)-f^{*}\leq\dfrac{1}{\mu}\|\grad f(x)\|^{2},\qquad x\in M.
\end{equation*}
\end{itemize}
\end{mdframed}

To the best of our knowledge, the inequality in A5 was first introduced by Polyak in \cite{polyak1963gradient} within the context of linear optimization. In this seminal work, the inequality played an important role in the asymptotic convergence analysis of the classical gradient method. 

\begin{lemma}
\label{lem:6}
Suppose that A1 and A5  hold and let $\left\{x_{k}\right\}_{k\geq 0}$ be generated by Algorithm 1. If 
\begin{equation}
\beta_{k+1}<\eta^{2}L^{2},\quad\text{for}\,\,k=0,\ldots,T-1
\label{eq:25.1}
\end{equation}
and
\begin{equation}
T\geq 1+\left[ \frac{\eta^{2}L^{2}}{\mu} \epsilon^{-1} +1\right]\log\left(\frac{\eta^{2}L^{2}}{\|\grad f(x_{0})\|^{2}}\right)
\label{eq:26}
\end{equation}
for some $\epsilon>0$, then $\min\left\{  f(x_{k}) - f^{*} \colon k=0,\ldots,T-1  \right\}\leq \epsilon$.
\end{lemma}
\begin{proof}
By \eqref{eq:r1}, we have
\begin{align}
\beta_{1}&=\|\grad f(x_{0})\|^{2}\nonumber\\
\beta_{2}&=\beta_{1}\left(1+\frac{\|\grad f(x_{1})\|^{2}}{\beta_{1}}\right)=\|\grad f(x_{0})\|^{2}\left(1+\frac{\|\grad f(x_{1})\|^{2}}{\beta_{1}}\right)\nonumber\\
\beta_{3}&=\beta_{2}\left(1+\frac{\|\grad f(x_{1})\|^{2}}{\beta_{2}}\right)=\|\grad f(x_{0})\|^{2}\left(1+\frac{\|\grad f(x_{1})\|^{2}}{\beta_{1}}\right)\left(1+\frac{\|\grad f(x_{2})\|^{2}}{\beta_{2}}\right)\nonumber\\
   &\vdots  \nonumber\\
\beta_{T}&=\|\grad f(x_{0})\|^{2}\prod_{k=1}^{T-1}\left(1+\frac{\|\grad f(x_{k})\|^{2}}{ \beta_{k}}\right).
\label{eq:28}
\end{align}
Using \eqref{eq:25.1}, \eqref{eq:28} and A5, it follows that
\begin{align*}
\eta^{2}L^{2}&> \beta_{T}=\|\grad f(x_{0})\|^{2}\prod_{k=1}^{T-1}\left[1+\frac{\|\grad f(x_{k})\|^{2}}{\beta_{k}}\right]\nonumber\\
                     &\geq  \|\grad f(x_{0})\|^{2}\prod_{k=1}^{T-1}\left[1+\frac{\mu(f(x_{k})-f^{*})}{\eta^{2}L^{2}}\right]\nonumber\\
                     &\geq  \|\grad f(x_{0})\|^{2}\prod_{k=1}^{T-1}\left[1+\frac{\mu}{\eta^{2}L^{2}}\min\left\{  f(x_{k}) - f^{*} \colon k=1,\ldots,T-1  \right\} \right]\nonumber\\
                     & =    \|\grad f(x_{0})\|^{2}\left[1+\frac{\mu}{\eta^{2}L^{2}} \min\left\{  f(x_{k}) - f^{*} \colon k=1,\ldots,T-1  \right\}\right]^{T-1}.
\end{align*}
Now, suppose by contradiction that $\min\left\{  f(x_{k}) - f^{*} \colon k=0,\ldots,T-1  \right\}>\epsilon$.
Then, applying the previous inequality and using the fact that the logarithm function is increasing, we obtain  
\begin{eqnarray*}  
\log\left(\frac{\eta^{2}L^{2}}{\|\grad f(x_{0})\|^{2}}\right)  
&>& (T - 1) \log\left(1 + \frac{\mu\epsilon}{\eta^{2}L^{2}}\right)  \\
&\geq & (T - 1) \dfrac{\frac{\mu\epsilon}{\eta^{2}L^{2}}}{1+ \frac{\mu\epsilon}{\eta^{2}L^{2}}}  
= (T - 1) \left[ \dfrac{\eta^{2}L^{2}}{\mu} \epsilon^{-1} + 1 \right]^{-1},  
\end{eqnarray*}  
which contradicts \eqref{eq:26}. This completes the proof. 
\end{proof}

\begin{lemma}
\label{lem:7}
Suppose that A1 and A5  hold  and let $\left\{x_{k}\right\}_{k\geq 0}$ be generated by Algorithm 1. If 
\begin{equation}
k_{0}=\min\left\{k\in\mathbb{N}\colon \beta_{k+1}\geq\eta^{2}L^{2}\right\}<+\infty
\label{eq:31}
\end{equation}
and
\begin{equation}
T_0\geq  \frac{\left| \log\left(\left[f(x_{0})-f^{*}+\frac{\eta^{4}L^{3}}{2\|\grad f(x_{0})\|^{2}}\right]\epsilon^{-1}\right) \right| }{ \left| \log\left(1-\frac{\mu\alpha_{\min}}{2}\right) \right| }
\label{eq:32}
\end{equation}
for some $\epsilon>0$ and for $\alpha_{\min}$ defined in (\ref{eq:18}), then
\begin{equation}
f(x_{k_{0}+ T_0})-f^{*}\leq\epsilon.
\label{eq:33}
\end{equation} 
\end{lemma}

\begin{proof}
By (\ref{eq:31}), (\ref{eq:r2}), and \eqref{eq:r1}, we have
$\alpha_{k}=\eta/\sqrt{ \beta_{k+1} }  \leq 1/L$ for all $k\geq k_{0}.$
Thus, by Lemma \ref{le:rlc} (ii), A5, and \eqref{eq:18}, we have
\begin{equation}\label{eq:proof.lem.PL.1}
f(x_{k}) - f(x_{k+1}) \;\geq\; \frac{\alpha_{k}}{2}\,\|\grad  f(x_{k})\|^{2} 
\;\geq\; \frac{\mu \alpha_{\min}}{2}\,\big(f(x_{k}) - f^{*} \big),\qquad \forall k\geq k_{0}.
\end{equation}
From this, it follows that
\begin{equation}\label{eq:term.btw.01}
1-\frac{\mu\alpha_{\min}}{2} \in (0,1).
\end{equation}
Furthermore, \eqref{eq:proof.lem.PL.1} implies that 
$$
f(x_{k+1})-f^{*} 
\leq  \left(1-\frac{\mu\alpha_{\min}}{2}\right) \left( f(x_{k})-f^{*} \right),\qquad \forall k\geq k_{0}.
$$
Hence,
\begin{equation}\label{eq:34}
f(x_{k+1})-f^{*} 
\leq \left(1-\frac{\mu\alpha_{\min}}{2}\right)^{k-k_0+1}\left( f(x_{k_0})-f^{*} \right),\qquad \forall k\geq k_{0}. 
\end{equation}
If $k_{0}\geq 1$, it follows from inequality \eqref{eq:extra1.1} in Lemma \ref{le:rlc} and from \eqref{eq:31} that
\begin{equation}
f(x_{k_{0}})-f^{*} \leq f(x_{0})-f^{*}+\dfrac{\eta^{4}L^{3}}{2\|\grad f(x_{0})\|^{2}}.
\label{eq:35}
\end{equation}
Clearly (\ref{eq:35}) is also true when $k_{0}=0$. If
$T_0 = 0$, then it follows from \eqref{eq:32} and \eqref{eq:35} that \eqref{eq:33} is true. Now consider the case $T_0 \geq 1$. By combining (\ref{eq:34}) with $k= k_0+T_0-1$ and (\ref{eq:35}), we obtain
\begin{equation}
f(x_{k_{0}+T_0})-f^{*} \leq\left(1-\frac{\mu\alpha_{\min}}{2}\right)^{T_0}\left(f(x_{0})-f^{*}+\dfrac{\eta^{4}L^{3}}{2\|\grad f(x_{0})\|^{2}}\right).
\label{eq:36}
\end{equation}
Thus, it follows from (\ref{eq:32}), \eqref{eq:term.btw.01} and (\ref{eq:36}) that (\ref{eq:33}) holds. Indeed, otherwise, we would have  
\begin{equation*}
\epsilon < \left(1-\frac{\mu\alpha_{\min}}{2}\right)^{T_0} \left(f(x_{0}) - f^{*} + \dfrac{\eta^{4}L^{3}}{2\|\grad f(x_{0})\|^{2}}\right),
\end{equation*}
which, by \eqref{eq:term.btw.01} and the properties of the logarithm, leads to 
$$
T_0 \left|\log\left(1-\frac{\mu\alpha_{\min}}{2}\right) \right|  
< \log\left(\left(f(x_{0}) - f^{*} + \dfrac{\eta^{4}L^{3}}{2\|\grad f(x_{0})\|^{2}}\right) \epsilon^{-1} \right),
$$
contradicting \eqref{eq:32}.
\end{proof}

\begin{theorem}
\label{thm:3}
Suppose that A1 and A5 hold  and let $\left\{x_{k}\right\}_{k\geq 0}$ be generated by Algorithm 1. For each $\epsilon>0$, define
$
T_{f}(\epsilon) = \inf \left\{ k \in \mathbb{N} : f(x_{k}) - f^{*} \leq \epsilon \right\}.
$
If $\| \grad f(x_{0})\|\geq\eta L$, then
\begin{equation}
T_{f}(\epsilon)< 1 + \frac{\left| \log\left(\left[f(x_{0})-f^{*}+\frac{\eta^{4}L^{3}}{2\|\grad f(x_{0})\|^{2}}\right]\epsilon^{-1}\right) \right|}{\left| \log\left(1-\frac{\mu\alpha_{\min}}{2}\right) \right|},
\label{eq:37}
\end{equation}
where $\alpha_{\min}$ is defined in (\ref{eq:18}). Otherwise, 
\begin{eqnarray}
T_{f}(\epsilon)&<& 1+\left[ \frac{\eta^{2}L^{2}}{\mu} \epsilon^{-1} +1\right]\log\left(\frac{\eta^{2}L^{2}}{\|\grad f(x_{0})\|^{2}}\right) \nonumber\\
& &+ 
\frac{\left| \log\left(\left[f(x_{0})-f^{*}+\frac{\eta^{4}L^{3}}{2\|\grad f(x_{0})\|^{2}}\right]\epsilon^{-1}\right) \right|}{\left| \log\left(1-\frac{\mu\alpha_{\min}}{2}\right) \right|}.
\label{eq:38}
\end{eqnarray}
\end{theorem}

\begin{proof}
First, let us consider the case $\|\grad f(x_{0})\| \geq \eta L$. 
Then, it follows from
\eqref{eq:r1} with $k=0$ that  
$\beta_1 = \|\grad f(x_{0})\|^2 \;\geq\; \eta^2 L^2$, which guarantees the equality 
\begin{equation}
k_{0} \;:=\; \inf\left\{ k \in \mathbb{N} \;:\; \beta_{k+1} \geq \eta^{2}L^{2} \right\} = 0.
\label{eq:39}
\end{equation}
If $T_{f}(\epsilon)=0$, then (\ref{eq:37}) holds. Therefore, suppose that $T_{f}(\epsilon)\geq 1$. In this case, by the definition of $T_{f}(\epsilon)$, we have
\begin{equation}
f(x_{T_{f}(\epsilon)-1})- f^{*}>\epsilon.
\label{eq:40}
\end{equation}
Thus, in view of (\ref{eq:39}) and (\ref{eq:40}), it follows from the contrapositive of Lemma \ref{lem:7} that we must have
\begin{equation*}
T_{f}(\epsilon)-1< \frac{\left| \log\left(\left[f(x_{0})-f^{*}+\frac{\eta^{4}L^{3}}{2\|\grad f(x_{0})\|^{2}}\right]\epsilon^{-1}\right) \right| }{\left| \log\left(1-\frac{\mu\alpha_{\min}}{2}\right) \right|},
\end{equation*} 
which establishes (\ref{eq:37}).

Now, suppose that $\|\grad f(x_{0})\|<\eta L$. If $T_{f}(\epsilon)=0$, then (\ref{eq:38}) holds. So, assume that $T_{f}(\epsilon)\geq 1$. Let us divide the rest of the analysis in two cases.
\\[0.2cm]
\noindent\textbf{Case 1:} $k_{0}:=\inf\left\{k\in\mathbb{N}\,:\,\beta_{k+1}\geq\eta^{2}L^{2}\right\}=+\infty$.
\\[0.2cm]
\noindent In this case, in particular, we have
\begin{equation}
\beta_{k+1}<\eta^{2}L^{2},\quad\text{for}\,\,k=0,\ldots,T_{f}(\epsilon)-1.
\label{eq:41}
\end{equation}
Moreover, by the definition of $T_{f}(\epsilon)$ we also have
\begin{equation}
\min\{ f(x_{k})-f^{*} \colon k =0,1,\ldots, T_{f}(\epsilon)-1 \} >\epsilon.
\label{eq:42}
\end{equation}
Thus, in view of (\ref{eq:41}) and (\ref{eq:42}), it follows from that contrapositive of Lemma \ref{lem:6} that we must have
\begin{equation*}
T_{f}(\epsilon)<1+\left[ \frac{\eta^{2}L^{2}}{\mu} \epsilon^{-1}+1\right]\log\left(\frac{\eta^{2}L^{2}}{\|\grad f(x_{0})\|^{2}}\right).
\end{equation*}
Therefore, (\ref{eq:38}) is true in this case.
\\[0.2cm]
\noindent\textbf{Case 2:} $k_{0} :=\inf\left\{k\in\mathbb{N}\,:\,\beta_{k+1}\geq\eta^{2}L^{2}\right\}<+\infty$.
\\[0.2cm]
\noindent 
From the definition of $k_{0}$ we have
\begin{equation}
\beta_{k+1}<\eta^{2}L^{2},\quad\text{for}\,\,k=0,\ldots,k_{0}-1.
\label{eq:43}
\end{equation}
Regarding the relation between $T_{f}(\epsilon)$ and $k_{0}$, there are only two possibilities.
\\[0.2cm]
\noindent\textbf{Subcase 2.1:} $T_{f}(\epsilon)\leq k_{0}$.
\\[0.2cm]
\noindent In this subcase, by (\ref{eq:43}) we have
\begin{equation*}
\beta_{k+1}<\eta^{2}L^{2},\quad\text{for}\,\,k=0,\ldots,T_{f}(\epsilon)-1.
\end{equation*}
Therefore, as in Case 1, we conclude that (\ref{eq:38}) is true.
\\[0.2cm]
\noindent\textbf{Subcase 2.2:} $T_{f}(\epsilon)=k_{0}+T_0$ for some $T_0\geq 1$.
\\[0.2cm]
In this case, in addition to (\ref{eq:43}), we also have
\begin{equation*}
\min\{ f(x_{k})-f^{*} \colon k =0,1,\ldots, k_0 \} \geq \min\{ f(x_{k})-f^{*} \colon k =0,1,\ldots, T_{f}(\epsilon)-1 \} >\epsilon.
\end{equation*}
Thus, by the contrapositive of Lemma \ref{lem:6} com $T=k_0+1$, it follows that
\begin{equation}
k_{0}<\left[ \frac{\eta^{2}L^{2}}{\mu} \epsilon^{-1}+1\right]\log\left(\frac{\eta^{2}L^{2}}{\|\grad f(x_{0})\|^{2}}\right).
\label{eq:44}
\end{equation}
If $T_0=1$, then it follows from (\ref{eq:44}) that
\begin{equation*}
T_{f}(\epsilon)=k_{0}+T_0<1+\left[ \frac{\eta^{2}L^{2}}{\mu} \epsilon^{-1}+1\right]\log\left(\frac{\eta^{2}L^{2}}{\|\grad f(x_{0})\|^{2}}\right),
\end{equation*}
and so (\ref{eq:38}) is true. On the other hand, if $T_0 \geq 2$, then
\begin{equation*}
f(x_{k_{0}+T_0-1})-f^{*}=f(x_{T_{f}(\epsilon)-1})-f^{*}>\epsilon.
\end{equation*}
Thus, by the contrapositive of Lemma \ref{lem:7} we must have
\begin{equation}
T_0-1< \frac{\left| \log\left(\left[f(x_{0})-f^{*}+\frac{\eta^{4}L^{3}}{2\|\grad f(x_{0})\|^{2}}\right]\epsilon^{-1}\right) \right|}{\left| \log\left(1-\frac{\mu\alpha_{\min}}{2}\right) \right|}.
\label{eq:45}
\end{equation}
By combining (\ref{eq:44}) and (\ref{eq:45}) with the fact that 
$T_{f}(\epsilon) = k_{0} + T_0$, we conclude that (\ref{eq:38}) also holds in this subcase.
\end{proof}

\subsubsection{A Class of Nonconvex PL functions on SPD matrices}

In this section, we provide a class of nonconvex functions that satisfy Assumption A5 on a particular Hadamard manifold. Let \( \mathbb{R}^{n\times n} \) be the set of real matrices of order \( n \times n \), \( \mathbb{P}^n \subset \mathbb{R}^{n\times n}  \) the set of symmetric matrices, and \( \mathbb{P}^n_{++} \subset \mathbb{R}^{n\times n}  \) the cone of symmetric positive definite matrices. Define
\begin{equation} \label{eq:metric}
\langle U, V \rangle_{X} := \operatorname{tr}(V X^{-1} U X^{-1}), \quad X \in \mathbb{P}^n_{++}, \; U, V \in \mathbb{P}^n,
\end{equation}
where \( \operatorname{tr}(\cdot) \) denotes the trace operator. It is well known that 
\( M = (\mathbb{P}^n_{++}, \langle \cdot, \cdot \rangle) \) 
is a Hadamard manifold (see, for example,~\cite[Theorem 1.2, Page 325]{Lang1999}), and that 
\( T_X M \) can be identified with \( \mathbb{P}^n \) for every \( X \in M \). The Riemannian gradient and Riemannian Hessian of \( f : \mathbb{P}_{++}^n \to \mathbb{R} \) are given, respectively, by
\begin{align} 
\mbox{grad} f(X)&=Xf'(X)X, \label{eq:Grad}\\
\mbox{hess}\,f(X)V&=Xf''(X)VX+\frac{1}{2}\left[  Vf'(X)X+  Xf'(X)V \right] \label{eq:Hess},
\end{align}
where \( V \in T_XM \), and \( f'(X) \) and \( f''(X) \) denote the Euclidean gradient and Hessian of \( f \) at \( X \), respectively, with respect to the Frobenius metric.


Consider the class of functions \( f : \mathbb{P}_{++}^n \to \mathbb{R} \) defined by
\begin{equation}  \label{eq:fpdm4}
f(X) = a \ln^4(\det(X)) - b \ln^3(\det(X)) - \frac{b^3}{a^2} \ln(\det(X)),
\end{equation}
where $a,b>0$.
Since
\begin{equation}\label{eq:Grad.eucli}
f'(X) = \left[4a\ln^3(\det(X)) - 3b\ln^2(\det(X)) - \frac{b^3}{a^2} \right] X^{-1},
\end{equation}
it follows from \eqref{eq:Grad} that
\begin{equation}\label{eq:Grad.riem.ex}
\grad f(X) = \left[4a\ln^3(\det(X)) - 3b\ln^2(\det(X)) - \frac{b^3}{a^2} \right] X.
\end{equation}
Therefore, the set of critical points of \( f \) is
$
\Omega \equiv \left\{ X \in {\mathbb P}^n_{++} \colon \det(X) = e^{b/a} \right\}.
$
Moreover, \eqref{eq:fpdm4} implies that \( f(X) = - b^4/a^3 \) for all \( X \in \Omega \). Given this and the coercivity of \( f \), we conclude that
$f^{*} = - b^4/a^3$. Consequently, using \eqref{eq:metric}, \eqref{eq:fpdm4}, and \eqref{eq:Grad.riem.ex}, along with appropriate algebraic manipulations, we obtain
\small
\begin{eqnarray*}
\frac{\| \grad f(X) \|^{2} }{f(X) - f^{*}} 
& = & \frac{ \left[ 4a\ln^3(\det(X)) - 3b\ln^2(\det(X)) - b^3/a^2 \right]^2n   }{ a \ln^4(\det(X)) 
- b \ln^3\left(\det(X)\right)-  (b^3/a^2)\ln\left(\det(X)\right)   + (b^4/a^3) } \\
& = &\frac{ \left[ \left[ \ln(\det(X)) -b/a \right] \left[ 4a\ln^2(\det(X)) + b\ln(\det(X)) + b^2/a \right] \right]^2 n  }
{ \left[ \ln(\det(X)) - b/a \right]^2 \left[ a\ln^2(\det(X))+b\ln\left(\det(X)\right)+b^2/a \right] } \\
& = & \frac{ 9na^2\ln^4(\det(X))   }{ a \ln^2(\det(X)) + b \ln\left(\det(X)\right)+  b^2/a } \\
& &+ n\left( 7a\ln^2(\det(X)) +b\ln(\det(X)) + b^2/a \right)  \\
&\geq& n\left( 7a\ln^2(\det(X)) +b\ln(\det(X)) + b^2/a \right) \geq (27nb^2)/(28a)
\end{eqnarray*}
\normalsize
for all \( X \in {\mathbb P}^n_{++} \setminus \Omega \), which shows that the function \( f \) defined in \eqref{eq:fpdm4} satisfies A5 for every \( 0 < \mu \leq (27nb^2)/(28a) \).
On the other hand, denoting the Euclidean norm by \( \| \cdot \|_e \), it follows from \eqref{eq:Grad.eucli} that
\small
\begin{multline*}
\frac{\|  f'(X) \|_{e}^{2} }{f(X) - f^{*}}  \\
= \left[ \frac{ 9a^2\ln^4(\det(X))   }{ a \ln^2(\det(X)) + b \ln\left(\det(X)\right)+  b^2/a } + 7a\ln^2(\det(X)) +b\ln(\det(X)) + b^2/a \right] \operatorname{tr}\left( X^{-2}\right), 
\end{multline*}
\normalsize
and, denoting by \( I_n \) the \( n \times n \) identity matrix, we obtain
\small
\begin{align*}
\inf_{X \in {\mathbb P}^n_{++} }\frac{\|  f'(X) \|_{e}^{2} }{f(X) - f^{*}} 
& = \lim_{t\to +\infty}  \frac{\|  f'(tI_n) \|_{e}^{2} }{f(tI_n) - f^{*}} \\
& =  \lim_{t\to +\infty} \left[ \frac{ 9a^2\ln^4( t^n )   }{ a \ln^2(t^n) + b \ln\left(t^n\right)+  b^2/a } + 7a\ln^2(t^n) +b\ln(t^n) + b^2/a \right] \frac{n}{t^2} \\
&= 0,
\end{align*}
\normalsize
which implies that there exists no \( \mu > 0 \) such that the function \( f \) defined in \eqref{eq:fpdm4} satisfies A5 in the Euclidean setting.

Now, observe that
\begin{multline*}
f''(X)V = \left[12a \ln^2(\det(X)) - 6b \ln(\det(X)) \right] \operatorname{tr}(X^{-1}V) X^{-1} \\
- \left[4a \ln^3(\det(X)) - 3b \ln^2(\det(X)) - \frac{b^3}{a^2} \right] X^{-1} V X^{-1},
\end{multline*}
for all \( X \in \mathbb{P}_{++}^n \) and \( V \in \mathbb{P}^n \). By combining this equality with \eqref{eq:Grad.eucli} and \eqref{eq:Hess}, we obtain
\[
\hess  f(X)V = \left[12a \ln^2(\det(X)) - 6b \ln(\det(X)) \right] \operatorname{tr}(X^{-1}V) X,
\]
for all \( X \in \mathbb{P}_{++}^n \) and \( V \in \mathbb{P}^n \), which implies that
$$
\langle \hess f(X)V, V \rangle = -3b^2/(4a) \|V\|^2 < 0
$$
for all \( X \in \left\{ X \in \mathbb{P}_{++}^n \colon \det(X) = e^{b/(4a)} \right\} \) and \( V \in \mathbb{P}^n \). Therefore,  \( f \) is not convex.

\section{Numerical Results}
\label{sec:numerics}

We evaluated the relative performance of Algorithm 1 by testing its Matlab implementation with $\eta=10$ (MAdaGrad) against the Riemannian Gradient Method with Armijo line search \cite{FerreiraLouzeiroPrudente2019} and the RWNGrad \cite{grapiglia2023adaptive}. The experiments were conducted on two classes of test problems. 

All implementations were performed in Matlab R2022a on a MacBook Pro equipped with an Apple M1 Pro processor and 16~GB of RAM. To ensure reproducibility, we fixed the randomness using Matlab's built-in function \texttt{rng(2025)}.

\subsection{Problem Class 1}
We considered class of problems of the form 
\begin{equation}
\min_{X\in\mathbb{P}_{++}^{n}}\,f\left(X\right)\equiv \ln\left(\det(X)\right)^{2}-\ln\left(\det(X)\right).
\label{eq:5.1}
\end{equation}
For $n=10$, the codes were run from 100 starting points, randomly generated with eigenvalues in the interval $(0,20)$. Following \cite{FerreiraLouzeiroPrudente2019,grapiglia2023adaptive}, each starting point $X_{0}$ was constructed as $X_{0}=Q^{T}\Gamma Q$, where $\Gamma$ is a diagonal matrix whose entries are independent random variables uniformly distributed in $(0,20)$, and $Q$ is obtained from the QR decomposition of a matrix with entries uniformly generated in $(0,1)$. Figure~\ref{fig:1} shows the performance profiles \cite{Dolan} with respect to CPU time for finding $X$ such that $\|\operatorname{grad} f(X)\|\leq 10^{-4}$, with each code allowed a maximum of $1,000$ iterations. As can be seen, MAdaGrad is significantly faster than the other two methods.
\newpage
\begin{figure}[h]
    \centering
    \includegraphics[width=0.5\textwidth]{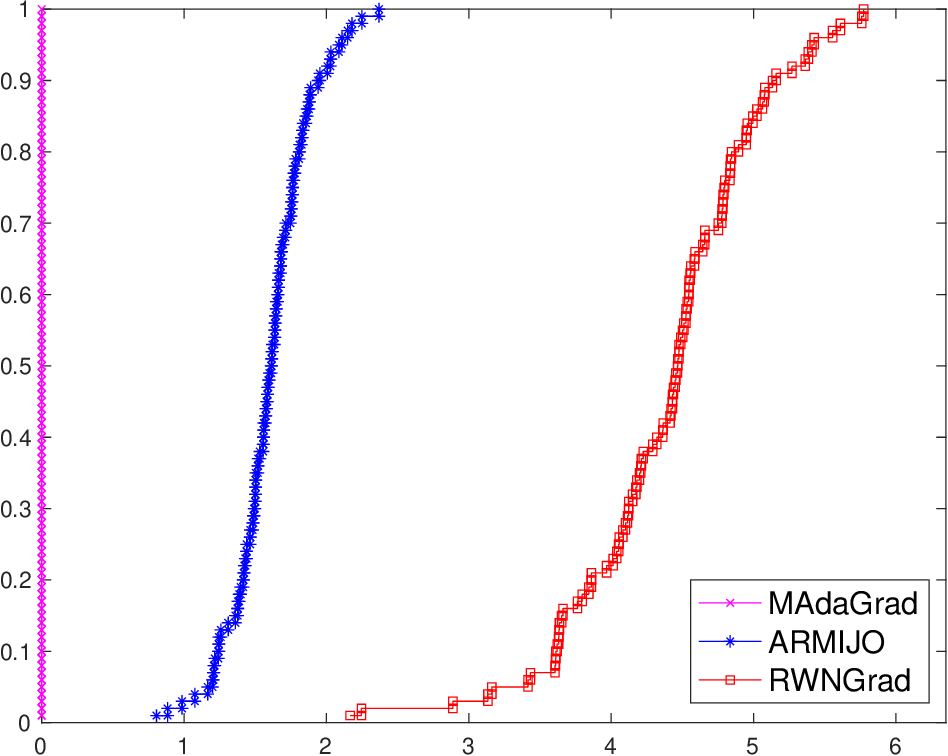}
    \caption{The Performance profiles (in $\log_{2}$ scale) with respect to CPU time for Problem 1. The magenta line corresponds to MAdaGrad, the blue line to ARMIJO, and the red line to RWNGrad.}
    \label{fig:1}
\end{figure}

\subsection{Problem Class 2}

We also considered the class of problems of the form
\begin{equation}
\min_{X\in\mathbb{P}_{++}^{n}}\, f(X)\equiv \frac{1}{2}\sum_{j=1}^{m}\left\|\ln\!\left(X^{-1/2}A_{j}X^{-1/2}\right)\right\|_{F}^{2},
\label{eq:5.2}
\end{equation}
for fixed $A_{1},\ldots,A_{m}\in\mathbb{P}_{++}^{n}$. For $n=20$ and $m=5$, we ran 100 test problems generated by randomly constructing 100 sets of matrices $A_{1},\ldots,A_{m}\in\mathbb{P}_{++}^{n}$. As in \cite{FerreiraLouzeiroPrudente2019,grapiglia2023adaptive}, each matrix $A_{j}$ was constructed as $A_{j}=Q_{j}^{T}\Lambda_{j}Q_{j}$, where $\Lambda_{j}$ is a diagonal matrix whose entries are independent random variables uniformly distributed in $(0,20)$, and $Q_{j}$ is obtained from the QR decomposition of a matrix with entries uniformly generated in $(0,1)$. For each problem, the initial point $X_{0}$ was chosen as
\begin{equation*}
    X_{0}=\exp\!\left(\frac{1}{m}\sum_{j=1}^{m}\ln(A_{j})\right).
\end{equation*}
Figure~\ref{fig:2} shows the performance profiles \cite{Dolan} with respect to CPU time for finding $X$ such that $\|\operatorname{grad} f(X)\|\leq 10^{-4}$, with each code allowed a maximum of $1,000$ iterations. Once again, our method MAdaGrad is considerably faster than the other two methods.
\newpage
\begin{figure}[h]
    \centering
    \includegraphics[width=0.5\textwidth]{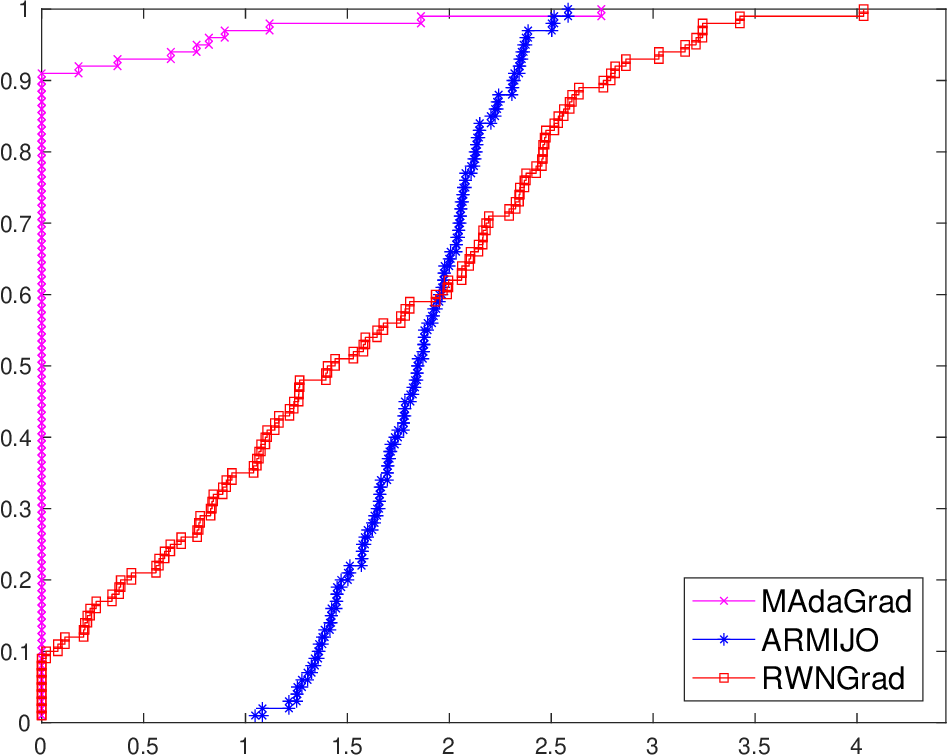}
    \caption{The Performance profiles (in $\log_{2}$ scale) with respect to CPU time for Problem 2. The magenta line corresponds to MAdaGrad, the blue line to ARMIJO, and the red line to RWNGrad.}
    \label{fig:2}
\end{figure}

\section{Conclusion}

In this paper, we have introduced \textsc{MAdaGrad}, a novel generalization of AdaGrad-Norm to Riemannian optimization. We established iteration complexity guarantees in several regimes: $\mathcal{O}(\varepsilon^{-2})$ for finding $\varepsilon$-stationary points under Lipschitz continuous Riemannian gradients; $\mathcal{O}(\varepsilon^{-1})$ for geodesically convex objectives on manifolds with sectional curvature bounded from below; and $\mathcal{O}(\log(\varepsilon^{-1}))$ under a global Polyak--\L{}ojasiewicz condition. Furthermore, we constructed nonconvex functions on the manifold $\mathbb{P}_{++}^{n}$ of symmetric positive definite matrices that satisfy the PL condition. Numerical experiments confirmed the efficiency of \textsc{MAdaGrad}, showing consistent improvements over Riemannian Steepest Descent with Armijo line-search \cite{FerreiraLouzeiroPrudente2019} and the RWNGrad method \cite{grapiglia2023adaptive} on optimization problems over $\mathbb{P}_{++}^{n}$.


\def\cprime{$'$} \def\cprime{$'$} \def\cprime{$'$}

%
%
%
%
%
%
%
\end{document}